\documentclass[11pt,twoside]{amsart}
\usepackage[russian,english]{babel}
\Russian

\usepackage{amssymb}
\usepackage{latexsym}

\usepackage{color}

\catcode`\@=11
\def\omathop#1#2#3{\let\temp=#1\def\letter{#2}
  \ifcat#3_ \let\next\@@olim\else\let\next\@olim\fi\next#3}
\def\@olim{\letter\text{-}\!\temp}
\def\@@olim_#1{\mathchoice{
   \setbox0=\hbox{$\displaystyle\letter\text{-}\!\temp\!\text{-}\letter$}
   \setbox2=\hbox{$\displaystyle\temp$}
   \setbox4=\hbox{$\scriptstyle#1$}
   \dimen@=\wd4 \advance\dimen@ by -\wd2 \divide\dimen@ by2
   \def\next{\letter\text{-}\!\temp_{\hbox to 0pt{\hss$\scriptstyle#1$\hss}}
     \hskip\dimen@}
   \ifdim\wd2>\wd4 \def\next{\@olim_{#1}}\fi
   \ifdim\wd4>\wd0 \def\next{\mathop{\llap{$\letter$-}\!\temp}\limits_{#1}}\fi
   \next}
   {\@olim_{#1}}{\@olim_{#1}}{\@olim_{#1}}}

\def\bolim{\omathop{\lim}{bo}}

\catcode`\@=13

\newcommand{\reduce}{\mskip-2mu}
\newcommand{\ls}{\reduce\left\bracevert\reduce\vphantom{X}}
\newcommand{\rs}{\reduce\vphantom{X}\reduce\right\bracevert\reduce}

\theoremstyle{plain}
\newtheorem{thm}{Theorem}[section]

\newtheorem{cor}[thm]{Corollary}
\newtheorem{lemma}[thm]{Lemma}

\theoremstyle{definition}
\newtheorem{definition}[thm]{Definition}

\numberwithin{equation}{section}

\begin{document}

\title{Narrow operators on lattice-normed spaces}

\author{M.~Pliev}

\address{South Mathematical Institute of the Russian Academy of Sciences\\
str. Markusa 22,
Vladikavkaz, 362027 Russia}

%\email{maratpliev@gmail.com}

%\centerline{\today}



\begin{abstract}
The aim of this article is to extend results of Maslyuchenko~O., Mykhaylyuk~V. Popov~M. about narrow operators on vector lattices. We give a new definition of a narrow operator where a vector lattice as the domain space of a narrow operator is replaced with a lattice-normed space. We prove that every $GAM$-compact $\text{(bo)}$-norm continuous linear operator from a Banach-Kantorovich space $V$ to a Banach lattice $Y$ is narrow. Then  we  show that, under some mild conditions, a continuous dominated operator is narrow if and only if its  exact dominant is.
\end{abstract}

\maketitle

%\tableofcontents

\section{Introduction}

\subsection{}
Today the theory of narrow operators is a very active area of Functional Analysis {\cite{B,Bi-1,Bi-2,F,Kad-2,Kad-3,Ma}}. Plichko and Popov were first $[20]$ who systematically studied this class of operators. It is worth remarking, however, that narrow operators have been studied in some particular cases by some others authors before this notion appeared. For example, Ghoussoub and Rosental $[8]$ have considered {``}norm-signed preserving operators{''} on $L_{1}[0,1]$, which are precisely the operators on $L_{1}[0,1]$ which are not narrow. On the other hand, Enflo and Starbird $[6]$ proved that if $T:L_{1}(\mu)\rightarrow L_{1}(\nu)$  is $L_{1}$-complementary singular (i.e. $T$ is invertible on no complemented subspace of $L_{1}(\mu)$isomorphic to $L_{1}(\mu)$) then $T$ is narrow. Johnson, Maurey, Schechtman and Tzafriri $[9]$ proved that every operator $T:L_{p}(\mu)\rightarrow L_{p}(\nu),1 < p<2$ which is $L_{p}$-complementary  singular is narrow. Later Kadets, Shvidkov and Werner had considered narrow operators in a different context $[13]$. Flores and Ruiz considered narrow operators from a K{\"{o}}the function space $E$ $[7]$. Finally, Maslyuchenko, Mykhaylyuk and Popov have considered a general vector-lattice approach to narrow operators $[18]$.

\subsection{}

In this seminal paper [18] the authors gave a new definition of a narrow operator.

\begin{definition} \label{def:0.2}
Let $E$ be an atomless order complete vector lattice, $X$ a Banach space.  A map $f:E\rightarrow X$ is called {\it narrow} if for every $x\in E_{+}$ and every $\varepsilon>0$ there exist some $y\in E$ such that $|y|=x$ and $\|f(y)\|<\varepsilon$. We  say that $f$ is {\it strictly narrow} if for every $x\in E_{+}$ there exists some $y\in E$ such that $|y|=x$ and $f(y)=0$.
\end{definition}

In  the same paper $[18]$ another definition of a narrow operator for the case when the range space is a vector lattice, was given.
Let $E,F$ be vector lattices with $E$ atomless. A linear operator  $T:E\rightarrow F$  is called {\it order narrow} if for every $x\in E_{+}$ there exists a net $(x_{\alpha})$ in $E$ such that $|x_{\alpha}|=x$ for each $\alpha$ and $Tx_{\alpha}\overset{(o)}\rightarrow 0$.

\subsection{}
In this paper we consider narrow operators in the framework of lattice-normed spaces. The notion of a lattice-normed space was introduced by Kantorovich in the first part of 20th century $[10]$. Later, Kusraev and his school had provided a deep theory. A detailed account the reader can find in $[15]$.

\section{Preliminaries}

The  goal of this section is to introduce some basic definitions and facts. General information on vector lattices,
Banach spaces and lattice-normed spaces the reader can find in the books $[1,2,15,16,17,19]$.

{\bf 1.\,1.}~Consider a vector space $V$ and a real  archimedean vector lattice
$E$. A map $\ls \cdot\rs:V\rightarrow E$ is a \textit{vector norm} if it satisfies the following axioms:
\begin{enumerate}
  \item[1)] $\ls v \rs\geq 0;$\,\, $\ls v\rs=0\Leftrightarrow v=0$;\,\,$(\forall v\in V)$.
  \item[2)] $\ls v_1+v_2 \rs\leq \ls v_1\rs+\ls v_2 \rs;\,\, ( v_1,v_2\in V)$.
  \item[3)] $\ls\lambda  v\rs=|\lambda|\ls v\rs;\,\, (\lambda\in\Bbb{R},\,v\in V)$.
\end{enumerate}
A vector norm is called \textit{decomposable} if
\begin{enumerate}
  \item[4)] for all $e_{1},e_{2}\in E_{+}$ and $x\in V$ from $\ls x\rs=e_{1}+e_{2}$ it follows that there exist $x_{1},x_{2}\in V$ such that $x=x_{1}+x_{2}$ and $\ls x_{k}\rs=e_{k}$, $(k:=1,2)$.
\end{enumerate}

A triple $(V,\ls\cdot\rs,E)$ (in brief $(V,E),(V,\ls\cdot\rs)$ or $V$ with default parameters omitted) is a \textit{lattice-normed space} if $\ls\cdot\rs$ is a $E$-valued vector norm in the vector space $V$. If the norm $\ls\cdot\rs$ is decomposable then the space $V$ itself is called decomposable. We say that a net $(v_{\alpha})_{\alpha\in\Delta}$ {\it $(bo)$-converges} to an element $v\in V$ and write $v=\bolim v_{\alpha}$ if there exists a decreasing net $(e_{\gamma})_{\gamma\in\Gamma}$ in $E$ such that $\inf_{\gamma\in\Gamma}(e_{\gamma})=0$ and for every $\gamma\in\Gamma$ there is an index $\alpha(\gamma)\in\Delta$ such that $\ls v-v_{\alpha(\gamma)}\rs\leq e_{\gamma}$ for all $\alpha\geq\alpha(\gamma)$. A net $(v_{\alpha})_{\alpha\in\Delta}$ is called \textit{$(bo)$-fundamental} if the net $(v_{\alpha}-v_{\beta})_{(\alpha,\beta)\in\Delta\times\Delta}$ $(bo)$-converges to zero. A lattice-normed space is called {\it $(bo)$-complete} if every $(bo)$-fundamental net $(bo)$-converges to an element of this space. Let $e$ be a positive element of a vector lattice $E$.
By $[0,e]$ we denote the set $\{v\in V:\,\ls v\rs\leq e\}$.
A set $M\subset V$ is called  $\text{(bo)}$-{\it bounded } if there exists
$e\in E_{+}$ such that  $M\subset[0,e]$. Every decomposable $(bo)$-complete lattice-normed space is called a {\it Banach-Kantorovich space} (a BKS for short).

\subsection{}
Let $(V,E)$ be a lattice-normed space.  A subspace $V_{0}$ of $V$ is called a $\text{(bo)}$-ideal of $V$ if for $v\in V$ and $u\in V_{0}$, from $\ls v\rs\leq\ls u\rs$ it follows that $v\in V_{0}$. A subspace $V_{0}$ of a decomposable lattice-normed space $V$  is a $\text{(bo)}$-ideal if and only if $V_{0}=\{v\in V:\,\ls v\rs\in L\}$, where $L$ is an order ideal in $E$ [15,\,2.1.6.1]. Let $V$ be a lattice-normed space and $y,x\in V$. If $\ls x\rs\bot\ls y\rs=0$ then we call the elements $x,y$ {\it disjoint} and write $x\bot y$. The equality $x=\coprod_{i=1}^{n}x_{i}$ means that $x=\sum_{i=1}^{n}x_{i}$ and $x_{i}\bot x_{j}$ if $i\neq j$. An element $z\in V$ is called a {\it component} or a \textit{fragment} of $x\in V$ if $0\leq \ls z\rs\leq\ls x\rs$ and $x\bot(x-z)$. Two fragments $x_{1},x_{2}$ of $x$  are called \textit{mutually complemented} or $MC$, in short, if $x=x_1+x_{2}$. The notations $z\sqsubseteq x$ means that $z$ is a fragment of $x$. According to [1,\,p.86] an element $e>0$ of a vector lattice $E$ is called an {\it atom}, whenever $0\leq f_{1}\leq e$, $0\leq f_{2}\leq e$ and $f_{1}\bot f_{2}$ imply that either $f_{1}=0$ or $f_{2}=0$. A vector lattice $E$ is  atomless if there is no atom $e\in E$.

The following object will be often used in different constructions below. Let $V$ be a lattice-normed space and $x\in V$. A sequence $(x_{n})_{n=1}^{\infty}$ is called a {\it disjoint tree} on $x$ if $x_{1}=x$ and $x_{n}=x_{2n}\coprod x_{2n+1}$ for each $n\in\Bbb{N}$. It is clear that all $x_{n}$ are fragments of $x$. All lattice-normed spaces below we consider to be decomposable.

\subsection{}

Consider some important examples of lattice-normed spaces. We begin with simple extreme cases, namely vector lattices and normed spaces. If $V=E$ then the modules of an element can be taken as its lattice norm: $\ls v\rs:=|v|=v\vee(-v);\,v\in E$. Decomposability of this norm easily follows from the Riesz Decomposition Property holding in every vector lattice. If $E=\Bbb{R}$ then $V$ is a normed space.

Let $Q$ be a compact  and let $X$ be a Banach space. Let $V:=C(Q,X)$ be the space of  continuous vector-valued functions from $Q$ to $X$. Assign $E:=C(Q,\Bbb{R})$. Given $f\in V$, we define its lattice norm by the relation $\ls f\rs:t\mapsto\|f(t)\|_{X}\,(t\in Q)$. Then $\ls\cdot\rs$ is a decomposable norm [15,\,lemma 2.3.2].

Let $(\Omega,\Sigma,\mu)$ be a $\sigma$-finite measure space,
let $E$ be an order-dense ideal in  $L_{0}(\Omega)$ and let $X$ be a Banach space.
By $L_{0}(\Omega,X)$ we denote the space of (equivalence classes of)  Bochner $\mu$-measurable vector functions
acting from $\Omega$ to $X$. As usual, vector-functions are equivalent if they have
equal values at almost all points of the set $\Omega$. If $\widetilde{f}$ is the coset of a measurable vector-function $f:\Omega\rightarrow X$ then $t\mapsto\|f(t)\|$,$(t\in\Omega)$ is a scalar measurable function whose coset is denoted by the symbol $\ls\widetilde{f}\rs\in L_{0}(\mu)$. Assign by definition
$$
E(X):=\{f\in L_{0}(\mu,X):\,\ls f\rs\in E\}.
$$
Then $(E(X),E)$ is a lattice-normed space with a decomposable norm [15,\,lemma 2.3.7.].
If $E$ is a Banach lattice then the lattice-normed space $E(X)$ is  a Banach space with respect to the norm $|\|f|\|:=\|\|f(\cdot)\|_{X}\|_{E}$.

\subsection{}

Let $E$ be a Banach lattice and let $(V,E)$ be a lattice-normed space. By definition,
$\ls x\rs\in E_{+}$ for every $x\in V$, and we can introduce some \textit{mixed norm} in $V$ by the formula
$$
\||x|\|:=\|\ls x\rs\|\,\,\,(\forall\, x\in V).
$$
The normed space $(V,\||\cdot|\|)$ is called a \textit{space with a mixed norm}.
In view of the inequality $|\ls x\rs-\ls y\rs|\leq\ls x-y\rs$ and monotonicity of the norm in $E$, we have
$$
\|\ls x\rs-\ls y\rs\|\leq\||x-y|\|\,\,\,(\forall\, x,y\in V),
$$
so a vector norm is a norm continuous operator from $(V,\||\cdot|\|)$ to $E$. A lattice-normed space $(V,E)$ is called
a \textit{Banach space with a mixed norm} if the normed space $(V,\||\cdot|\|)$ is complete with respect to the norm convergence.

\subsection{}

Consider lattice-normed spaces $(V,E)$ and $(W,F)$, a linear operator $T:V\rightarrow W$ and a positive operator
$S\in L_{+}(E,\,F)$. If the condition
$$
\ls Tv\rs\leq
S\ls v\rs;\,(\forall\, v\in V)
$$
is satisfied then we say that $S$ \textit{dominates} or
\textit{majorizes} $T$ or that $S$ is \textit{dominant}
or {majorant} for $T$.
In this case $T$ is called a \textit{dominated} or
\textit{majorizable} operator.  The set of all dominants of the operator $T$
is denoted by $\text{maj}(T)$. If there is the least element in
$\text{maj}(T)$ with respect to the order induced by $L_{+}(E,F)$ then
it is called the {\it least} or the {\it exact dominant} of $T$ and it is denoted by
$\ls T\rs$. The set of all dominated operators from $V$ to $W$ is denoted by $M(V,W)$.
Denote by $E_{0+}$ the conic hull of the set $\ls V\rs=\{\ls v\rs:\,v\in V\}$, i.e., the set of elements of the form $\sum_{k=1}^{n}\ls v_{k}\rs$, where $v_{1},\dots,v_{n}\in V$, $n\in\Bbb{N}$.

\begin{lemma}[\cite{Ku}, 4.1.2,\,4.1.5.] \label{le:1}
Let $(V,E),(W,F)$ be lattice-normed spaces. Suppose $V$ is decomposable and $F$ is order complete. Then every
dominated operator has the exact dominant $\ls T\rs$.
The exact dominant of an arbitrary operator $T\in M(V,W)$ can be calculated by the following formulas:
$$
\ls T\rs(e)
=\sup\left\{\sum\limits_{i=1}^n\ls Tv_i\rs:
\sum\limits_{i=1}^n\ls v_i\rs= e, \,e\in E_{0+}\right\};
$$
$$
\ls T\rs(e)=\sup\{\ls T\rs(e_{0}):\,e_{0}\in E_{0+};\, e_{0}\leq e\}
(e\in E_{+});
$$
$$
\ls T\rs(e)=\ls T\rs(e_+)+\ls T\rs(e_-),\,
(e\in E).
$$
\end{lemma}

{\bf Acknowledgment.}~I am very grateful to professor Mikhail Popov  for his valuable remarks and great help. I am also grateful to the referees for their useful  suggestions.

\section{Definition and some properties of narrow operators}
\label{sec2}

In this section we introduce a new class of operators in lattice-normed spaces and describe some of their properties.

\begin{definition} \label{def:nar1}
Let $(V,E)$ be a lattice-normed space,  $X$ a Banach space and suppose that $E$ is atomless. An operator  $T:V\rightarrow X$ is called \textit{narrow}, if for every $u\in V,\,\varepsilon>0$ there exist two $MC$ fragments $u_{1},u_{2}$ of $u$ such that $\|T(u_{1}-u_{2})\|<\varepsilon$. If for every $u\in V$ there exist two $MC$ fragments $u_{1},u_{2}$ of  the $u$ such that   $T(u_{1}-u_{2})=0$ for then the operator $T$ is called  \textit{strictly narrow}.
\end{definition}

The set of all narrow operators from a lattice-normed space $(V,E)$ to a Banach space $X$ we denote by $\mathcal{N}(V,X)$.

\begin{lemma} \label{le:2}
If a lattice-normed space $(V,E)$ coincides with $(E,E)$ then definitions \ref{def:0.2} and \ref{def:nar1} are equivalent.
\end{lemma}

\begin{proof}
Let $X$ be a Banach space and let $T:E\rightarrow X$ be a narrow operator in accordance with Definition \ref{def:nar1}. Consider an element  $e\in E_{+}$ and $\varepsilon>0$. Then there exist two $MC$ fragments $e_{1},\,e_{2}$ of $e$ such that $\|T(e_{1}-e_{2})\|<\varepsilon$. Then for $y=e_{1}-e_{2}$ one has that $\|Ty\|<\varepsilon$, that is, Definition \ref{def:0.2} for $T$ is satisfied.

Now we prove the inverse assertion. Let $T$ be a narrow operator in accordance with Definition \ref{def:0.2}. Fix any $x\in E$ and $\varepsilon>0$. Then $x=x_{+}-x_{-}$ and there exist two elements $x_{1}'$ and $x_{2}'$ such that $|x_{1}'|=x_{+},\,\|Tx_{1}'\|<\frac{\varepsilon}{2}$ and $|x_{2}'|=x_{-},\,\|Tx_{2}'\|<\frac{\varepsilon}{2}$. We consider new elements: $e_{1}:=x_{1}'\vee 0$,\, $e_{2}:=-x_{1}'\vee 0$ and $f_{1}:=x_{2}'\vee 0$,\, $f_{2}:=-x_{2}'\vee 0$. So, we have
two pairs of $MC$ fragments   $e_{1},e_{2}$ of $x_{+}$ and  $f_{1},f_{2}$ of $x_{-}$ such that the following inequalities  hold
$$
\|T(e_{1}-e_{2})\|<\frac{\varepsilon}{2};\,
\|T(f_{1}-f_{2})\|<\frac{\varepsilon}{2}.
$$
Then $x_{1}:=e_{1}-f_{2}$ and $x_{2}:=e_{2}-f_{1}$ are $MC$ fragments of  $x$, and
$$
\|T(x_{1}-x_{2})\|=\|T(e_{1}-f_{2}-e_{2}+f_{1})\|<
\|T(e_{1}-e_{2})\|+\|T(f_{1}-f_{2})\|<\varepsilon.
$$
\end{proof}

Let $(V,E)$ be a lattice-normed space and let $(W,F)$ be a Banach space with a mixed norm. An operator $T:V\rightarrow W$ is called  \textit{order narrow} if for every $u\in V$ there exists a net  $(v_{\alpha})_{\alpha\in\Lambda}$ where every element $v_{\alpha}$ is a difference  $u_{\alpha}^{1}-u_{\alpha}^{2}$ of two $MC$ fragments of $u$ such that $T(v_{\alpha})\overset{(bo)}\rightarrow 0$.

\begin{lemma} \label{le:3}
Let $(V,E)$ be a lattice-normed space and let $(W,F)$ be a Banach space with a mixed norm. Then every narrow operator $T:V\rightarrow W$ is order narrow.
\end{lemma}

\begin{proof}
We consider an element $u\in V$. Let $\varepsilon_{n}:=\frac{1}{2^{n}}$  and let $u_{n}^{1},u_{n}^{2}$ be  $MC$ fragments~of the element $u$ and $v_{n}=(u_{n}^{1}-u_{n}^{2})$, $|\|Tv_{n}|\|\leq\varepsilon_{n}$.  We set  $e_{n}=\sum\limits_{k=n}^{\infty}\ls Tv_{k}\rs$, $e_{n}\in F_{+}$, $e_{n}\downarrow 0$ and the following estimate holds $\ls Tv_{n}\rs\leq e_{n}$. Thus, $Tv_{n}\overset{(bo)}\rightarrow 0$.
\end{proof}

The sets of narrow and order narrow operators coincide if a vector lattice $F$ is good enough.

\begin{lemma}\label{le:4}
Let  $(V,E)$ and $(W,F)$ be the same as in $3.3$ and let $F$ be a Banach lattice with order continuous norm. Then a linear operator  $T:V\rightarrow W$ is order narrow if and only if $T$ is narrow.
\end{lemma}

\begin{proof}
Let $T$ be an order narrow operator. Then for every  $u\in V$ there exist a net
$(v_{\alpha})_{\alpha\in\Lambda};\,v_{\alpha}=u_{\alpha}^{1}-u_{\alpha}^{2}$,
where $u_{\alpha}^{1}$ and $u_{\alpha}^{2}$ are $MC$ fragments of $u$ and $Tv_{\alpha}\overset{(bo)}\rightarrow 0$. Fix any $\varepsilon>0$. Using the fact that the norm in $F$ is order continuous we can find $\alpha_{0}\in\Lambda$ such that $\|\ls Tv_{\alpha}\rs\|<\varepsilon$ for every $\alpha\geq\alpha_{0}$. In view of Lemma $3.3$, the converse is true.
\end{proof}

The following lemma will be useful later.

\begin{lemma} \label{le:5}
Let $(V,E)$ $(J,F_{1})$ and $(W,F)$ be lattice-normed spaces, $E$ an atomless vector lattice, $J$ a $\text{(bo)}$-ideal of $W$ and $F_{1}$ an order ideal of $F$.  If a linear dominated operator  $T:V\rightarrow J$ is order narrow then $T:V\rightarrow W$ is order narrow as well. Conversely, if a dominated linear operator  $T:V\rightarrow J$ is such that $T:V\rightarrow W$ is an order narrow then so is $T:V\rightarrow J$.
\end{lemma}

\begin{proof}
The first part is obvious. Let $T:V\rightarrow J$ be a dominated  operator such that $T:V\rightarrow W$ is order narrow. For any $v\in V$ we choose a net $(v_{\alpha})_{\alpha\in\Lambda}$ where every element $v_{\alpha}$ is a difference  $u_{\alpha}^{1}-u_{\alpha}^{2}$ of two $MC$  fragments of $u$ such that   $T(v_{\alpha})\overset{(bo)}\rightarrow 0$, that is, $\ls Tv_{\alpha}\rs\leq y_{\alpha}\downarrow 0$ for some net $(y_{\alpha})_{\alpha\in\Delta}\subset F$. Using the fact that the operator $T$ is dominated and its exact dominant is $\ls T\rs:E\rightarrow F_{1}$, one has that
$$
\ls Tv_{\alpha}\rs\leq\ls T\rs\ls v\rs=g\in F_{1}.
$$
Hence, $\ls Tv_{\alpha}\rs\leq z_{\alpha}\downarrow 0$ where $z_{\alpha}:=g\wedge y_{\alpha}$. So, we have that $(z_{\alpha})_{\alpha\in\Delta}\subset F_{1}$. Thus, the net $(Tv_{\alpha})_{\alpha\in\Delta}$ $\text{(bo)}$-converges to $0$ in $(J,F_{1})$.
\end{proof}

\section{Narrow and $GAM$-compact operators}
\label{sec4}

In this section we investigate connections between narrow and $GAM$-compact operators.

Let $(V,E)$ be a lattice-normed space and let $Y$ be a Banach space. A linear operator $T:V\rightarrow Y$ is called  {\it $(\text{bo})$-norm continuous} whenever it sends every $\text{(bo)}$-convergent net in $V$ to a norm convergent net in $Y$. A linear operator $T:V\rightarrow Y$ is called  {\it generalized $AM$-compact} or {\it $GAM$-compact} for short if for every $\text{(bo)}$-bounded set  $M\subset V$ its image $T(M)$ is a relatively compact set in $Y$. Consider some examples.

{\it Example~1.}~In a particular case when $V=E$ the sets of $GAM$-compact and $AM$-compact operators  from $V$ to $Y$ are equal.

{\it Example~2.}~If $E=\Bbb{R}$ then $V$ is a normed space and the sets of $GAM$-compact and compact operators from $V$ to $Y$ are equal.

{\it Example~3.}~Let $X,Y$ be Banach spaces and let $(\Omega,\Sigma,\mu)$ be a finite measure space. The space $L_{1}(\mu,X)$ is the space  $\mu$-Bochner integrable functions on $\Omega$ with values in $X$, and  $L_{\infty}(\mu,X)$ is the space of $X$-valued $\mu$-Bochner integrable functions on $\Omega$ that are essentially bounded. A function $g\in L_{\infty}(\mu,\mathcal{L}(X,Y))$ is said to have its essential range in the {\it uniformly compact operators} if there is a compact set $C$ in $Y$ such that $g(\omega)x\in C$ for almost all $\omega\in\Omega$ and $x\in X$, $\|x\|\leq 1$. An operator $T:L_{1}(\mu,X)\rightarrow Y$ is called {\it representable measurable kernel} if there is a bounded measurable function $g:\Omega\rightarrow\mathcal{L}(X,Y)$ such that
$$
Tf=\int_{\Omega}fg\,d\mu \,\,\,\,\,\,\,\,\text{for all}\,\,f\in L_{1}(\mu,X) .
$$

\begin{thm}[\cite{Ke}, Theorem~2]\label{t:1}
Let $X$ be a Banach space such that $X^{\star}$ has the Radon-Nikod\'{y}m property. Then there is an isometric isomorphism between the space of compact operators $K(L_{1}(\mu,X),Y)$ and the subspace of $L_{\infty}(\mu,K(X,Y))$ consisting of those functions whose essential range is in the uniformly compact operators. In fact, $T \in K(L_{1}(\mu,X),Y)$ and $g \in L_{\infty}(\mu,K(X,Y))$ are in correspondence if and only if
$$
T(f)=\int_{\Omega}g(\omega)f(\omega)\,d\mu(\omega)\,\,\, \text{for all} \,\,\,f\in L_{1}(\mu,X).
$$
\end{thm}

\begin{lemma}
Let $X$ be a Banach space such that $X^{\star}$ has the Radon-Nikod\'{y}m property.  Then every  representable measurable kernel operator $T:L_{1}(\mu,X)\rightarrow Y$ whose kernel  have the essential range in the  uniformly compact operators is $GAM$-compact.
\end{lemma}

\begin{proof}
Let $M\subset V$ be a $(\text{bo})$-bounded set. Then there exists $e\in L_{1}(\mu)$, $e\geq 0$, such that $\ls v\rs\leq e$ for every $v\in M$. This implies that the set $M$ is norm bounded in $L_{1}(\mu, X)$
$$
\|v\|_{L_{1}(\mu, X)}=\int_{\Omega}\ls v\rs d\,\mu=
\int_{\Omega}\| v(\cdot)\|_{X} d\,\mu\leq\int_{\Omega}e d\,\mu=r
$$
By Theorem $4.1$, the lemma is proved.
\end{proof}

\begin{lemma}
Let  $(V,E)$ be a Banach space with a mixed norm and let $X$ be a Banach space. Then every  $GAM$-compact operator $T:V\rightarrow X$ is norm bounded.
\end{lemma}

\begin{proof}
If $T$ were unbounded then we would have a sequence $(v_{n})\subset V$ with $\|\ls v_{n}\rs\|\leq\frac{1}{2^{n}}$ and $\|Tv_{n}\|\rightarrow\infty$. The sequence $(v_{n})$  is order bounded by the element $e\in E_{+}$, $e:=\sum\limits_{n=1}^{\infty}\ls v_{n}\rs$. Hence, the sequence  $(Tv_{n})$ is relatively compact, -- a contradiction.
\end{proof}

\begin{cor}
Let  $(V,E)$ be a Banach space with a mixed norm, $E$ an order continuous Banach lattice and $X$ a Banach space. Then every  $GAM$-compact operator $T:V\rightarrow X$ is $\text{(bo)}$-norm continuous.
\end{cor}

The following lemma is well known [11,\,p.14].

\begin{lemma}
Let  $(x_{i})_{i=1}^{n}$ be a finite collection of vectors in a finite dimensional normed space  $X$ and let $(\lambda_{i})_{i=1}^{n}$ be a collection of reals with $0\leq\lambda_{i}\leq 1$ for each $i$. Then there exists a collection $(\theta_{i})_{i=1}^{n}$ of numbers $\theta_{i}\in\{0,1\}$ such that
$$
\Big\|\sum_{i=1}^{n}(\lambda_{i}-\theta_{i})x_{i}\Big\|\leq\frac{\text{dim}(X)}{2}\max_{i}\|x_{i}\|
$$
\end{lemma}

For the rest of the section $(V,E)$ is assumed to be a Banach-Kantorovich space, $E$ is an atomless order complete vector lattice,
$X$ is a Banach space and $T:V\rightarrow X$ is $\text{(bo)}$-norm continuous linear operator.

\begin{lemma}
Let  $x\in V$ . Then there exist two $MC$ fragments $x_{1},x_{2}$ of $x$ such that $\|Tx_{1}\|-\|Tx_{2}\|=0$.
\end{lemma}

\begin{proof}
Fix any couple $x_{1},x_{2}$ of $MC$ fragments of $x$. If $\|Tx_{1}\|-\|Tx_{2}\|=0$ then there is nothing to prove. Let $\|Tx_{1}\|-\|Tx_{2}\|>0$. Consider the partially ordered set
$$
D:=\{y\sqsubseteq x_{1}:\|T(x_{1}-y)\|-\|T(x_{2}+y)\|\geq 0\}
$$
where $y_{1}\leq y_{2}$ if and only if $y_{1}\sqsubseteq y_{2}$. If $B\subset D$ is a chain then $y^{\star}=\vee B\in D$ by the $\text{(bo)}$-norm continuity of $T$. By the Zorn lemma, there is a maximal element $y_{0}\in D$. Now we show  $\|T(x_{1}-y_{0})\|-\|T(x_{2}+y_{0})\|=0$. Suppose on the  contrary that
$$
\alpha=\|T(x_{1}-y_{0})\|-\|T(x_{2}+y_{0})\|> 0.
$$
Since $E$ is atomless, we can choose  a fragment $0\neq y\sqsubseteq(x_{1}-y_{0})$ with $|\|Ty|\|<\frac{\alpha}{3}$. Since $y_{0}\bot y$, $y_{0}+y\sqsubseteq x_{1}$ and
$$
\|T(x_{1}-y_{0}-y)\|-\|T(x_{2}+y_{0}+y)\|\geq
$$
$$
\geq\|T(x_{1}-y_{0})\||-\|Ty\|-\|T(x_{2}+y_{0})\|-\|Ty\|>\frac{\alpha}{3},
$$
that contradicts the maximality of $y_{0}$.
\end{proof}

\begin{lemma}
Let $v\in V$ and $(v_{n})_{n=1}^{\infty}$ be a disjoint tree on $v$. If  $\|Tx_{2n}\|=\|Tx_{2n+1}\|$ for every $n\leq 1$ then
$$
\lim\limits_{m\rightarrow\infty}\max\limits_{2^{m}\leq i< 2^{m+1}}\|Tv_{i}\|=0
$$
\end{lemma}

\begin{proof}
Put $\gamma_{m}=\max\limits_{2^{m}\leq i<2^{m+1}}\|Tv_{i}\|$  and $\varepsilon=\limsup\limits_{m\rightarrow\infty}\gamma_{m}$. Suppose on the contrary that $\varepsilon>0$. Then for each $n\in\Bbb{N}$ we set
$$
\varepsilon_{n}=\limsup\limits_{m\rightarrow\infty}\max\limits_{2^{m}\leq i<2^{m+1},\,v_{i}\sqsubseteq v_{n}}\|Tv_{n}\|.
$$
Hence, for each $m\in\Bbb{N}$ one has
$$
\max\limits_{2^{m}\leq i<2^{m+1}}\varepsilon_{i}=\varepsilon.\,\,\,\,\,(\star)
$$
Now we are going to construct a sequence of mutually disjoint elements $(v_{n_{j}})_{j=1}^{\infty}$ such that $\|Tv_{n_{j}}\|\leq\frac{\varepsilon}{2}$. At the first step we choose $m_{1}$ such that $\max\limits_{2^{m_{1}}\leq i<2^{m_{1}+1}}\|Tv_{i}\|\geq\frac{\varepsilon}{2}$. In accordance with $(\star)$, we choose $i_{1}$, $2^{m_{1}}\leq i_{1}<2^{m_{1}+1}$ so that $\varepsilon_{m_{1}}=\varepsilon$. Using $\|Tv_{2n}\|=\|Tv_{2n+1}\|$, we choose $n_{1}\neq i_{2}$, $2^{m_{1}}\leq i<2^{m_{1}+1}$ so that $\|Tv_{n_{1}}\|\leq\frac{\varepsilon}{2}$. At the second step we choose $m_{2}>m_{1}$ so that
$$
\max\limits_{2^{m}\leq i<2^{m+1},\,v_{i}\sqsubseteq v_{i_{1}}}\|Tv_{i}\|\geq\frac{\varepsilon}{2}.
$$
In accordance with  $(\star)$, we choose $i_{2}$, $2^{m_{2}}\leq i_{2}<2^{m_{2}+1}$ so that $\varepsilon_{i_{2}}=\varepsilon$. Then we choose $m_{2}\neq i_{2}$, $2^{m_{2}}\leq i_{2}<2^{m_{2}+1}$ so that $\|Tv_{m_{2}}\|\geq\frac{\varepsilon}{2}$. Proceeding further, we construct the desired sequence. Indeed, $\|Tv_{m_{i}}\|\geq\frac{\varepsilon}{2}$ by the construction and mutually disjoint for $v_{m_{l}},v_{m_{j}}$, $j\neq l$ is guaranteed by the condition $m_{j}\neq i_{j}$. The elements  $v_{m_{j+l}}$ are fragments of $v_{i_{j}}$ which are disjoint to $v_{m_{j}}$.
\end{proof}

The following lemma indicates that  operators with finite dimensional range are also narrow.

\begin{lemma}
If $X$ is finite dimensional, then $T$ is narrow.
\end{lemma}

\begin{proof}
Fix any $v\in V$, $\varepsilon>0$ and $\text{dim}(X)=\gamma$. Using Lemma $4.6$ we construct a disjoint tree $(v_{n})$ on $v$ with $\|Tv_{2n}\|=\|Tv_{2n+1}\|$ for all $n\in\Bbb{N}$. By lemma $4.7$ we choose $m$ such that $\gamma\alpha_{m}<\varepsilon$ where $\alpha_{m}:=\max\limits_{2^{m}\leq i<2^{m+1}}\|Tv_{i}\|$. Then using Lemma $4.5$, we choose numbers $(\lambda_{i})_{i=1}^{n}$, $\lambda_{i}\in\{0,1\}$ for $i=2^{m},\dots, 2^{m+1}-1$ so that
$$
\Big\|\sum_{i=2^{m}}^{2^{m+1}-1} \Bigl(\frac{1}{2}-\lambda_{i} \Bigr)Tv_{i}\Big\|
\leq\frac{\gamma}{2}\max\limits_{2^{m}\leq i<2^{m+1}}\|Tv_{i}\|=
\frac{\gamma}{2}\alpha_{m}<\frac{\varepsilon}{2}.
$$
Now we consider the element $w=2 \Bigl(\sum\limits_{i=2^{m}}^{2^{m+1}-1}(\frac{1}{2}-\lambda_{i})v_{i} \Bigr)$. On the other hand, $w=\sum\limits_{i=2^{m}}^{2^{m+1}-1}u_{i}$ where  $u_{i}$ are disjoint, and $u_{i}\in\{(\pm v_{i})_{i=2^{m}}^{2^{m+1}-1}\}$. Then there exist two $MC$-fragments $v_{1}$ and $v_{2}$ of $v$ such that $w=v_{1}-v_{2}$  and $\|Tw=T(v_{1}-v_{2})\|<\varepsilon$.
\end{proof}

Now we are ready to prove the main result of this section.

\begin{thm}
Let $(V,E)$ be a Banach-Kantorovich space, $E$ an atomless order complete vector lattice,
$X$ a Banach space. Then every $GAM$-compact $\text{(bo)}$-norm continuous linear operator  $T:V\rightarrow X$ is narrow.
\end{thm}

\begin{proof}
It is well known that if $H$ is a relatively compact subset of $l_{\infty}(D)$ for some infinite set $D$ and $\varepsilon>0$ is an arbitrary positive number then there exists a finite rank operator $S\in l_{\infty}(D)$ such that $\|x-Sx\|\leq\varepsilon$ for every $x\in H$. So, we may consider $X$ as a subspace of some $l_{\infty}(D)$ space
$$
X \hookrightarrow X^{\star\star} \hookrightarrow l_{\infty}(B_{X^{\star}})=l_{\infty}(D)=W.
$$
By the notation $\hookrightarrow$ we mean isometric embedding. Fix any $v\in V$ and $\varepsilon>0$. Since $T$ is a $GAM$-compact operator, $K=\{Tu:\ls u\rs\leq\ls v\rs\}$ is relatively compact in $X$ and hence, in $W$. Then there exist a finite dimensional operator
$S\in\mathcal{L}(W)$ such that $\|x- Sx\|\leq\frac{\varepsilon}{2}$ for every $x\in K$. Then $R=S\circ T$ is a $(\text{bo})$-norm continuous finite dimensional operator. By Lemma $4.8$, there exist two $MC$ fragments $v_{1},v_{2}$ of $v$ such that $\|R(v_{1}-v_{2})\|<\frac{\varepsilon}{2}$. Thus,
$$
\|T(v_{1}-v_{2})\|=\|T(v_{1}-v_{2})+S(T(v_{1}-v_{2}))-S(T(v_{1}-v_{2}))\|=
$$
$$
=\|T(v_{1}-v_{2})+R(v_{1}-v_{2})-S(T(v_{1}-v_{2}))\|\leq
$$
$$
\leq\|R(v_{1}-v_{2})\|+\|T(v_{1}-v_{2})-S(T(v_{1}-v_{2}))\|
<\frac{\varepsilon}{2}+\frac{\varepsilon}{2}=\varepsilon.
$$
\end{proof}

\section{Dominated narrow operators}
\label{sec5}

In this section we investigate some properties of the dominated  narrow  operators. Observe that for every $x,y\in L_{1}(\nu)$ the following equality  holds
$$
\|x-y\|=\||x|-|y|\|+\|x\|+\|y\|-\|x+y\|.\,\,\,(\star)
$$

\begin{thm}
Let $E,F$ be  order complete vector lattices such that $E$ is atomless,  $F$ an ideal of some order continuous Banach lattice  and $(V,E)$ a Banach-Kantorovich space. Then every $(bo)$-continuous dominated linear operator $T:V\rightarrow F$ is  order narrow if and only if $\ls T\rs$ is.
\end{thm}

\begin{proof}
First we mention that a Banach lattice $F$ is a lattice-normed space and the vector norm $\ls\cdot\rs$ coincides with the module map $|\cdot|$. Now we prove the theorem for $F=L_{1}(\nu)$. By Lemma $3.4$, instead of order narrowness we will consider narrowness. Assume first that $T$ is narrow.  Fix any $e\in E_{+}$ and $\varepsilon>0$. Let  $v\in V$ and $\ls v\rs=e$. Since
$$
\left\{\sum\limits_{i=1}^n|Tv_i|:
\sum\limits_{i=1}^n\ls v_i\rs= e;\,\ls v_{i}\rs\bot\ls v_{j}\rs;\,i\neq j;\,n\in\Bbb{N}\right\}
$$
is an increasing net, using the order continuity of $L_{1}(\nu)$, we can choose a finite collection $\{v_{1},\dots,v_{n}\}\subset V$ with
$$
e=\bigsqcup\limits_{i=1}^{n}\ls v_{i}\rs
$$
and
$$
\Big\|\ls T\rs(e)-\sum\limits_{i=1}^{n}|Tv_{i}|\,\Big\|<\varepsilon.
$$
Let $v_{i}=u_{i}\coprod w_{i}$ be a decomposition into a sum of $MC$ fragments $u_{i}$ and $w_{i}$. Then we have
$$
0\leq\ls T\rs(e)-\sum\limits_{i=1}^{n}(|Tu_{i}|+|Tw_{i}|)\leq
$$
$$
\leq\ls T\rs(e)-\sum\limits_{i=1}^{n}|Tv_{i}|
$$
Since
$$
e=\sum\limits_{i=1}^n\ls v_i\rs=\sum\limits_{i=1}^n(\ls u_i\rs+\ls w_{i}\rs),
$$
we have that
$$
\ls T\rs(e)=\ls T\rs(\sum\limits_{i=1}^n(\ls u_i\rs+\ls w_{i}\rs)
=\sum\limits_{i=1}^n(\ls T\rs\ls u_i\rs+\ls T\rs\ls w_{i}\rs).
$$
Since $\ls T\rs\ls u_{i}\rs-|Tu_{i}|$ and $\ls T\rs\ls w_{i}\rs-|Tw_{i}|$ are positive elements of $L_{1}(\nu)$ for every $i\in\{1,\dots,n\}$, the sum of their norms equals the norm of their sum.
Thus, we obtain
$$
\sum\limits_{i=1}^{n}(\|\ls T\rs\ls u_i\rs-|Tu_{i}|\,\|+\|\ls T\rs\ls w_i\rs-|Tw_{i}|\,\|)=
$$
$$
=\|\ls T\rs(e)-\sum\limits_{i=1}^{n}(|Tu_{i}|+|Tw_{i}|)\|\leq
$$
$$
\leq\|\ls T\rs(e)-\sum\limits_{i=1}^{n}|Tv_{i}|\,\|<\varepsilon.
$$
For each $i=1,\dots,n$ we represent $v_{i}=u_{i}\coprod w_{i}$ so that $u_{i},w_{i}\in V$ and $\|Tw_{i}-Tu_{i}\|<\frac{\varepsilon}{n}$.
Then putting $u=\coprod\limits_{i=1}^{n}u_{i}$, $w=\coprod\limits_{i=1}^{n}w_{i}$, $f_{1}=\ls u\rs$, $f_{2}=\ls w\rs$  and using inequality $(\star)$, we  obtain
$$
\|\ls T\rs f_{1}-\ls T\rs f_{2}\|\leq\sum\limits_{i=1}^{n}\|\ls T\rs\ls u_{i}\rs-\ls T\rs\ls w_{i}\rs\|\leq
$$
$$
\leq\sum\limits_{i=1}^{n}\|\,|Tu_{i}|-| Tw_{i}|\,\|+
\sum\limits_{i=1}^{n}(\|\ls T\rs\ls u_{i}\rs-|Tu_{i}|\,\|+
\|\ls T\rs\ls w_{i}\rs-|Tw_{i}|\,\|)\leq
$$
$$
\leq\sum\limits_{i=1}^{n}\| Tu_{i}- Tw_{i}\|+\varepsilon<2\varepsilon.
$$
Using the arbitrariness of $e \in E_{+}$ and $\varepsilon>0$, and the fact that $f_{1}, f_{2}$ are two $MC$ fragments of $e$, we deduce that $\ls T\rs$ is a narrow operator.

Now let $\ls T\rs$ be a narrow operator, $v$ an arbitrary element of $V$, $\varepsilon>0$, $\ls v\rs=e$, $e=\coprod\limits_{i=1}^{n}\ls v_{i}\rs$,\,$v_{i}\in V;\,\forall i\in\{1,\dots,n\}$ and again
$$
\Big\|\ls T\rs(e)-\sum\limits_{i=1}^{n}|Tv_{i}|\,\Big\|<\varepsilon.
$$
For each $i=1,\dots,n$ we decompose $v_{i}=f^{1}_{i}\coprod f^{2}_{i}$ such that
$$
\Big\|\ls T\rs\ls f_{1}^{i}\rs-\ls T\rs\ls f_{i}^{2}\rs\Big\|<\frac{\varepsilon}{n}
$$
and let $f^{1}=\coprod\limits_{i=1}^{n}f_{1}^{i}$ and $f^{2}=\coprod\limits_{i=1}^{n}f_{2}^{i}$. Taking into account that the $L_{1}$-norm of a sum of positive elements equals the sum of their norm, we obtain
$$
\sum\limits_{i=1}^{n}|Tf_{i}^{j}|\leq\sum\limits_{i=1}^{n}\ls T\rs\ls f_{i}^{j}\rs;\,j\in\{1,2\};
$$
$$
\sum\limits_{i=1}^{n}\ls T\rs\ls f_{i}^{1}\rs+\sum\limits_{i=1}^{n}\ls T\rs\ls f_{i}^{2}\rs=\ls T\rs(e);
$$
$$
\Big\|\sum\limits_{i=1}^{n}|Tf_{i}^{1}|+|Tf_{i}^{2}|\,\Big\|\leq\|\ls T\rs (e)\|.
$$
Then, using again inequality $(\star)$, we obtain
$$
\|Tf_{1}-Tf_{2}\|\leq
\sum\limits_{i=1}^{n}\|Tf_{i}^{1}-Tf_{i}^{2}\|
=\sum\limits_{i=1}^{n}\|\,|Tf_{i}^{1}|-|Tf_{i}^{2}|\,\|+
$$
$$
+\sum\limits_{i=1}^{n}\|\,|Tf_{i}^{1}|+|Tf_{i}^{2}|\,\|-\sum\limits_{i=1}^{n}\|\,|Tv_{i}|\,\|=
$$
$$
=\sum\limits_{i=1}^{n}\|\,|Tf_{i}^{1}|-|Tf_{i}^{2}|+\ls T\rs\ls f_{i}^{1}\rs-\ls T\rs\ls f_{i}^{1}\rs+
\ls T\rs \ls f_{i}^{2}\rs-\ls T\rs\ls f_{i}^{2}\rs\|+
$$
$$
+\sum\limits_{i=1}^{n}\|\,|Tf_{i}^{1}|+|Tf_{i}^{2}|\,\|-\sum\limits_{i=1}^{n}\|\,|Tv_{i}|\,\|\leq
$$
$$
\leq\sum\limits_{i=1}^{n}\|\ls T\rs\ls f_{i}^{1}\rs-\ls T\rs\ls f_{i}^{2}\rs\|+
\sum\limits_{i=1}^{n}\|\ls T\rs\ls f_{i}^{1}\rs-|Tf_{i}^{1}|\,\|+
$$
$$
+\sum\limits_{i=1}^{n}\|\ls T\rs\ls f_{i}^{2}\rs-|Tf_{i}^{2}|\,\|+
\|\ls T\rs(e)\|-\sum\limits_{i=1}^{n}\|\,|Tv_{i}|\,\|
$$
Finally, using the fact that
$$
\|\ls T\rs(e)\|-\sum\limits_{i=1}^{n}\|\,|T v_{i}|\,\|=
\Big\|\ls T\rs(e)-\sum\limits_{i=1}^{n}|T v_{i}|\,\Big\|<\varepsilon
$$
we obtain that $\|Tf_{1}-Tf_{2}\|< 3\varepsilon$. Since $f_{1}$ and $f_{2}$ are $MC$ fragments of $v$, this proves that $T$ is narrow.

Now we consider the general case. Since $F$ is an ideal of some order continuous Banach lattice $H$, we have by Lemma $3.5$ that $T:V\rightarrow F$ is order narrow if and only if $ T:V\rightarrow H$ is.

We consider $T:V\rightarrow H$ and $\ls T\rs:E\rightarrow H$. Fix any $v\in V$. By $E_{1}$ and $H_{1}$ we denote the principal bands in $E$ and $H$ generated by $\ls v\rs$ and $\ls T\rs\ls v\rs$ respectively. Using the fact that Boolean algebras of bands $\mathcal{B}(V)$ and $\mathcal{B}(E)$ are isomorphic [15,\,2.1.2.1], we denote $V_{1}:=\gamma(E_{1})$. Here $\gamma:\mathcal{B}(E)\rightarrow\mathcal{B}(V)$ is a boolean isomorphism. Denote by $T_{1}$ the restriction of $T$ to $V_{1}$. Operator $\ls T_{1}\rs$ coincides with the restriction $\ls T\rs$ to $E_{1}$. So $H_{1}$ is an order continuous Banach lattice with weak unit $\ls T\rs\ls v\rs$. Then by [17,Theorem\,1.b.14] there exists a probability space $(\Omega,\Sigma,\nu)$ and an ideal $H_{2}$ of $L_{1}(\nu)$ such that $H_{1}$ is isomorphic to $H_{2}$. Let $S:H_{1}\rightarrow H_{2}$ be a lattice isomorphism. Then we set $T_{2}=S\circ T_{1}$. Moreover, $\ls T_{2}\rs=S\circ\ls T_{1}\rs$. By  our previous consideration, $T_{2}:V_{1}\rightarrow H_{2}$ is order narrow if and only if $\ls T_{2}\rs:E_{1}\rightarrow H_{2}$ is. Thus, we have proved that $T_{2}$ is order narrow if and only if $\ls T_{2}\rs$ is.

Let $T:V\rightarrow H$ be order narrow. Fix an arbitrary $v\in V$. Since $T_{1}$ is order narrow, so is $T_{2}$ and hence $\ls T_{2}\rs$. So there  exists a net  $(v_{\alpha})_{\alpha\in\Lambda}\subset V_{1}$ where every element $v_{\alpha}$ is a difference  $u_{\alpha}^{1}-u_{\alpha}^{2}$ two $MC$  fragments of $v$ such that   $\ls T_{2}\rs(\ls v_{\alpha}\rs)\overset{(o)}\rightarrow 0$. Therefore, $\ls T_{1}\rs(\ls v_{\alpha}\rs)\overset{(o)}\rightarrow 0$ and hence $\ls T\rs(\ls v_{\alpha}\rs)\overset{(o)}\rightarrow 0$. So, $\ls T\rs:E\rightarrow H$ is order narrow.
\end{proof}

Let  $(V,E),(W,F)$ be lattice-normed spaces and let $M(V,W)$ be the space of dominated operators from $V$ to $W$. The band generated by all lattice homomorphisms from $E$ to $F$ we denote $\mathcal{H}(E,F)$. Then, $\mathcal{H}(V,W):=\{T\in M(V,W):\,\ls T\rs\in\mathcal{H}(E,F)\}$.

\begin{thm}
Let $E,F$ be order complete vector lattices such that $E$ is atomless,  $F$ an ideal of some order continuous Banach lattice  and let $(V,E)$ be a Banach-Kantorovich space. Then every $(bo)$-continuous dominated linear operator $T:V\rightarrow F$ is uniquely represented in the form  $ T=T_{h}+T_{n}$, where $T_{h}\in \mathcal{H}(V,F)$ and $T_{n}$ is a $(bo)$-continuous order narrow operator.
\end{thm}

\begin{proof}
By $[15,4.2.1]$, the set $M(V,E)$ with the mapping $p: M(V,E)\rightarrow L_{+}(E,F)$ is a lattice-normed space. Here $p(T)=\ls T\rs$ for every
$T\in M(V,E)$. By  $[15,\,4.2.6]$, the vector norm $p:M(V,F)\rightarrow L_{+}(E,F)$ is decomposable. This means that for every dominated operator $T:V\rightarrow F$ and its exact dominant $\ls T\rs:E\rightarrow F$ the following statement hold:
$$
\ls T\rs=S_{1}+S_{2}\Rightarrow\exists T_{1},T_{2}\in M(V,F);
$$
$$
\ls T_{1}\rs=S_{1};\,\ls T_{2}\rs=S_{2};\,0\leq S_{1},S_{2};\,S_{1}\bot S_{2}.
$$
Fix an arbitrary $(bo)$-continuous dominated operator $T:V\rightarrow F$. By theorem
$[15,4.3.2]$, every dominated operator $T:V\rightarrow F$ is $(bo)$-continuous if and only if $\ls T\rs:E\rightarrow F$ is $(o)$-continuous. Hence, $\ls T\rs$ is $(o)$-continuous. Then by  $[18,\,11.7]$ the positive $(o)$-continuous operator $\ls T\rs$ is uniquely represented as a sum $\ls T\rs=S_{D}+S_{N}$ when $S_{D}$ is a $(o)$-continuous operator, $S_{D}\in\mathcal{H}(E,F)$ and $S_{N}$ is a $(o)$-continuous narrow operator. Then we obtain
$$
\ls T\rs=S_{D}+S_{N}\Rightarrow\exists T_{1},\,T_{2}\in M(V,F);\,
$$
$$
\ls T_{1}\rs=S_{N};\,\ls T_{2}\rs=S_{D}.
$$
Finally, by Theorem $5.1$, the proof is completed.
\end{proof}

\begin{thm}
Let $E(\mu)$ be a K{\"{o}}the function space over probability space
$(\Omega,\Sigma,\mu)$, with atomless measure $\mu$, such that $E(\mu)'$ is order continuous, $F$ is an order continuous Banach lattice  and $X$ is a Banach space. Then for every $(bo)$-continuous dominated narrow linear operator $T:E(X)\rightarrow F$ the inclusion $[0,\ls T\rs]\subset \mathcal{N}(E(X),F)$ holds.
\end{thm}

\begin{proof}
By theorem $[7,\,3.15]$, for every positive narrow operator $S:E(\mu)\rightarrow F$, such that $E(\mu)'$ is order continuous, the inclusion $[0,S]\subset \mathcal{N}(E(X),F)$ holds. Now by Theorem $5.1$ the proof is completed.
\end{proof}

\end{document}